\documentclass[10pt,reqno]{amsart}
\usepackage{amssymb}
\usepackage{mathrsfs}
\usepackage[pdftex,colorlinks=true,citecolor=blue,linkcolor=black,urlcolor=blue]{hyperref}
\usepackage[all]{xy}

\newcommand{\id}{{\operatorname{id}}}
\newcommand{\Tot}{{\operatorname{\rm Tot}}}

\newcommand{\ra}{\rightarrow}

\newcommand{\sm}{\boldsymbol{Sm}}

\newcommand{\blank}{\underline{\;\;}}
\newcommand{\ds}{\displaystyle}
\newcommand{\ub}{\underline}
\newcommand{\pt}{\mathrm{pretr}}

\newcommand{\mc}{\mathcal}

\newcommand{\by}[1]{\overset{#1}{\to}}
\newcommand{\longby}[1]{\overset{#1}{\longrightarrow}}
\newcommand{\iso}{\by{\sim}}
\newcommand{\eff}{\mathrm{eff}}
\renewcommand{\phi}{\varphi}

\newcommand{\dg}{{dg_e}}

\newcommand{\Z}{\mathbb{Z}}

\begin{document}
\begin{abstract}	
Recently, Levine   constructed a DG category
whose homotopy category is equivalent to the full subcategory of motives
over a base-scheme $S$ generated by the motives of smooth projective $S$-schemes, 
assuming that $S$ is itself smooth over a perfect field. 
In his construction, the tensor structure required $\mathbb{Q}$-coefficients. The author has previously shown how to provide a tensor structure on the homotopy category mentioned above, when $S$ is semi-local and essentially smooth over a field of characteristic zero, extending Levine's tensor structure with $\mathbb{Q}$-coefficients.
In this article, it is shown that, under these conditions, the fully faithful functor $\rho_S$ that Levine constructed from his category of smooth motives to the category $DM_S$ of motives over a base (defined by Cisinski-D\'{e}glise) is a tensor functor.
\end{abstract}

\title{Tensor functor from Smooth Motives to motives over a base}
\author{Anandam Banerjee}
\date{\today}
\address{
Mathematical Sciences Institute,\endgraf
Australian National University,\endgraf
Canberra, ACT 0200,\endgraf
Australia}
\email{anandam.banerjee@anu.edu.au}
\maketitle

\setcounter{tocdepth}{1}
\setcounter{secnumdepth}{3}
\theoremstyle{definition}
\newtheorem{dfn}{Definition}[section]
\theoremstyle{plain}
\newtheorem{thmintro}{Theorem}
\newtheorem{corintro}{Corollary}
\newtheorem{thm}{Theorem}[section]
\newtheorem{lem}[thm]{Lemma}
\newtheorem{prop}[thm]{Proposition}
\newtheorem{cor}[thm]{Corollary}
\theoremstyle{remark}
\newtheorem{clm}{Claim}
\newtheorem{rem}[dfn]{Remark}
\numberwithin{equation}{section}

\section*{Introduction}
Although this program has not been realized, Voevodsky has constructed a triangulated category of geometric motives over a perfect field, which has many of the properties expected of the derived category of the conjectural abelian category of motives. The construction of the triangulated category of motives has been extended by Cisinski-D\'{e}glise (\cite{cd}) to a triangulated category of motives over a base-scheme $S$, denoted $DM_S$. Hanamura (\cite{hana}) has also constructed a triangulated category of motives over
a field, using the idea of a ``higher correspondence", with morphisms built out of Bloch's cycle complex.
Recently, Bondarko (in \cite{bondarko}) has refined Hanamura's idea and limited it to smooth projective varieties to construct a DG category of motives. Assuming resolution of singularities, the homotopy category of this DG category is equivalent to 
Voevodsky's category of effective geometric motives. 
Soon after, Levine (in \cite{smmot}) extended this idea to construct a DG category of ``smooth motives''
over a base-scheme $S$ generated by the motives of smooth projective $S$-schemes, 
where $S$ is itself smooth over a perfect field. Its homotopy category is equivalent to the full subcategory of  Cisinski-D\'{e}glise category of effective motives over $S$ generated by the smooth
projective $S$-schemes.
Both these constructions lack a tensor structure in general. However, passing to $\mathbb{Q}$-coefficients, Levine replaced the cubical construction with alternating cubes, which yields a tensor structure on his DG category.

In \cite{tenssmmot}, a pseudo-tensor structure is constructed on a  DG category $\dg Cor_S$ which induces a tensor structure on the homotopy category of DG complexes,  such  that, in case $S$ is semi-local and essentially smooth over a field of characteristic zero, it
induces a tensor structure on the category of smooth motives over $S$. It is proved that
\begin{thmintro}\label{int} Suppose $S$ is semi-local and essentially smooth over a field of characteristic zero. Then, 
there is a tensor structure on the category $SmMot^{\mathrm{eff}}_{gm}(S)$ of smooth effective geometric motives over $S$ making it into a tensor triangulated category. 
\end{thmintro}
In \cite{smmot}, Levine constructed a functor from his category of smooth motives to the Cisinski-D\'{e}glise category of motives over $S$
\[\rho_S:SmMot_{gm}(S)\to DM_S. \]
In this article, we show that
\begin{thmintro}\label{tf} If $S$ is semi-local and essentially smooth over a field of characteristic zero, then $\rho_S$ is a tensor functor with respect to the tensor structure on $SmMot_{gm}(S)$ defined in Theorem~\ref{int}.
\end{thmintro}
This is done by first defining a morphism $\rho_S(M)\otimes\rho_S(N)\to \rho_S(M\otimes N)$ in $DM_S$, for objects $M,N$ in  $SmMot_{gm}(S)$, and then showing that this is an isomorphism for $M=M_S(X)$ and $N=M_S(Y)$, where $X,Y$ are smooth projective schemes and $M_S$ is the functor $\mathbf{Proj}_S\to SmMot_{gm}(S)$. We also show that
\begin{corintro}
Under the same conditions on $S$, we have an exact duality
\[ SmMot_{gm}^{op}(S)\to SmMot_{gm}(S). \]
\end{corintro}

We begin with proving some general results on pseudo-tensor DG categories and pseudo-tensor functors, and then prove Theorem~\ref{tf} over the next two sections.

\section{Generalities on pseudo-tensor categories}
In this section, we extend some results from \cite[\S~2.3]{tenssmmot}. Let $(\mc{A}, P^\mc{A})$ and $(\mc{B}, P^\mc{B})$ be pseudo-tensor DG categories.
\begin{dfn}
A DG functor $X:\mc{A}\to\mc{B}$ is called a pseudo-tensor functor if $X$ induces a map of complexes
\[P^\mc{A}_I(\{E_i\},F)^*\longrightarrow P^\mc{B}_I(\{X(E_i)\},X(F))^*. \]
which is compatible with composition as given by the following commutative diagram:
\begin{equation}
\label{eq:pstfunc}
\begin{xy}
\xymatrix{ P^\mathcal{A}_I(\{F_i\},G)^*\otimes\bigotimes_{i\in I}P^\mathcal{A}_{J_i}(\{E_j\}_{j\in J_i},F_i)^*\ar[d]\ar[r] & P^\mathcal{A}_J(\{E_j\},G)^*\ar[d]\\
P^\mathcal{B}_I(\{X(F_i)\},X(G))^*\otimes\bigotimes_{i\in I}P^\mathcal{B}_{J_i}(\{X(E_j)\}_{j\in J_i},X(F_i))^*\ar[r] & P^\mathcal{B}_J(\{X(E_j)\},X(G))^*
}
\end{xy}\end{equation} and $X(\id_{E})=\id_{X(E)}$.

Furthermore, let $(\mc{A}, P^\mc{A})$ and $(\mc{B}, P^\mc{B})$ be {\em homotopy tensor categories}, that is, their pseudo-tensor structure induce a tensor structure on their respective homotopy categories. $X$ would be called a {\em homotopy
tensor functor}, if, in addition, $H^0X:H^0\mc{A}\to H^0\mc{B}$ is a tensor functor with respect to the tensor structures on $H^0\mc{A}$ and $H^0\mc{B}$ induced by $P^\mc{A}$ and $P^\mc{B}$ respectively.
\end{dfn}

\begin{dfn}
A functor $X:\mc{A}\to \mc{B}$ between tensor categories $\mc{A}$ and $\mc{B}$ is called a {\em lax tensor functor}, if for every $E,F\in ob(\mc{A})$, there is a morphism \[
X(E)\otimes_{\mc{B}}X(F) \longby{\theta_{E,F}} X(E\otimes_{\mc{A}}F)\]
in $\mc{B}$, which is associative, commutative and unital (see \cite[Part II, Chap.~1, \S 1.3.7]{Mxmot}. In \cite{Mxmot}, Levine calls such a functor a pseudo-tensor functor -- we would not use this terminology to avoid confusion).
\end{dfn}

\begin{rem}
Note that a pseudo-tensor functor between two tensor categories (with pseudo-tensor structures induced by the respective tensor structures) is a lax tensor functor. \end{rem}
Thus, if $(\mc{A}, P^\mc{A})$ and $(\mc{B}, P^\mc{B})$ are homotopy tensor categories and $X:\mc{A}\to\mc{B}$ a pseudo-tensor DG functor, 
then $H^0X:H^0\mc{A}\to H^0\mc{B}$ is a { lax tensor functor}.

The set of small homotopy tensor categories form a category whose morphisms are homotopy tensor functors. It is shown in \cite[\S~2.3]{tenssmmot} that if $\mc{A}$ is a homotopy tensor category, so is $\mc{A}^\pt$. Here, we are going to show that the assignment $\mc{A}\mapsto\mc{A}^\pt$ is an endofunctor of the category of small homotopy tensor categories.
\begin{prop}\label{prop:homtens}
{\it i.} If $X:\mc{A}\to\mc{B}$ is a pseudo-tensor DG functor between pseudo-tensor DG categories $\mc{A}$ and $\mc{B}$, then there is an induced pseudo-tensor DG functor
$X^\pt:\mc{A}^\pt\to\mc{B}^\pt$.\\

\noindent {\it ii.} Further, if $\mc{A}$ and $\mc{B}$ are homotopy tensor categories, then $H^0X^\pt:H^0\mc{A}^\pt\to H^0\mc{B}^\pt$ is a lax tensor functor.\\

\noindent {\it iii.} Also, if in addition, $X:\mc{A}\to\mc{B}$ is a homotopy tensor functor, so is $X^\pt:\mc{A}^\pt\to\mc{B}^\pt$.
\end{prop}

\begin{proof}
{\it i.} Given a DG functor $X:\mc{A}\to\mc{B}$, we define a functor $X^\pt:\text{Pre-Tr}(\mc{A})\to\text{Pre-Tr}(\mc{B})$ as follows:\\
Let $\mc{E}=\{E_i,\,e_{ij}:E_j\to E_i\}$ be an object in $\text{Pre-Tr}(\mc{A})$. We define $X^\pt(\mc{E}):={\{X(E_i),\,X(e_{ij}):X(E_j)\to X(E_i)\}}$ and for a morphism $\phi:\mc{E}\to\mc{F}$ in $\text{Pre-Tr}(\mc{A})$, $X(\phi):={\{X(\phi_{ij}):X(E_j)\to X(F_i)\}}$. $X^\pt$ is a DG functor since so is $X$. Clearly, by definition, we have
\[ X^\pt(\mc{E}[n])=X^\pt(\mc{E})[n];\quad X^\pt(Cone(\phi))=Cone(X^\pt(\phi)). \]
Thus, restricting to $\mc{A}^\pt$, we get a functor \[ X^\pt:\mc{A}^\pt\to\mc{B}^\pt. \]

Now, for objects $\mc{E}^1,\ldots,\mc{E}^n,\mc{F}$ in $\mc{A}^\pt$, and $I=\{1,\ldots,n\}$, the pseudo-tensor structure on $\mc{A}^\pt$ is defined as:
\[ P^{\mc{A}^\pt}_I(\{\mc{E}^i\},\mc{F}):=\bigoplus_{j,k} P^\mc{A}_I(\{E^i_j\},F_k)[j-k]. \]
Thus, since $X$ is a pseudo-tensor functor, we have the map
\[
\begin{xy}\xymatrix{ P^{\mc{A}^\pt}_I(\{\mc{E}^i\},\mc{F}) \ar@{=}[d] \ar@{-->}[r] & P^{\mc{B}^\pt}_I(\{X^\pt(\mc{E}^i)\},X^\pt(\mc{F})) \ar@{=}[d] \\
\bigoplus_{j,k} P^\mc{A}_I(\{E^i_j\},F_k)[j-k] \ar[r] & \bigoplus_{j,k} P^\mc{B}_I(\{X(E^i_j)\},X(F_k))[j-k]}
\end{xy}
\]
It also follows that the required commutativity of the square \eqref{eq:pstfunc} is satisfied, thus proving {\em i}.

{\it ii.} Let $\mc{E},\,\mc{F}\in Ob(H^0\mc{A}^\pt)$. We first show that
\begin{clm} We have the following isomorphism in $H^0\mc{A}^\pt$:
\[\mathcal{E}\otimes\mc{F}\iso\{(\oplus_{i+j=l}E_i\otimes F_j)_l,g_{kl}:\oplus_{i+j=l}E_i\otimes F_j\to\oplus_{i+j=k}E_i\otimes F_j\}
\] where $g_{kl}:=\oplus_i\left((-1)^{(l-k+1)(l-i)}e_{(i+k-l),i}\otimes\id_{F_{l-i}} + (-1)^i\id_{E_i}\otimes f_{(k-i),(l-i)}\right)$.
\end{clm}
\begin{proof}[Proof of Claim]
We make use of the fact that $\mc{A}^\pt$ is generated by $i(\mc{A})$ by taking translations and cones and proceed by  induction. For objects $E,F$ in $\mc{A}$, and any $\mc{G}$ in $\mc{A}^\pt$, we have
\begin{align*}
Hom_{H^0\mc{A}^\pt}(i(E)\otimes i(F),\mc{G})& \iso H^0P^{\mc{A}^\pt}(\{i(E),i(F)\},\mc{G})\\
& \iso \bigoplus_kH^0P^\mc{A}(\{E,F\},G_k)[-k]\\
& \iso \bigoplus_kH^0Hom_\mc{A}(E\otimes F,G_k)[-k]\\
& \iso Hom_{H^0\mc{A}^\pt}(i(E\otimes F),\mc{G})
\end{align*}
This implies $i(E)\otimes i(F)\iso i(E\otimes F)$ in $H^0\mc{A}^\pt$. Thus, $(i(E)\otimes i(F))_0=E\otimes F$ and $(i(E)\otimes i(F))_l=0$ for $l\neq 0$. Also $(i(E)\otimes i(F))_{ij}=0$.

Also, note that it follows from the definition of the pseudo-tensor structure on $\mc{A}^\pt$, $\mc{E}[n]\otimes\mc{F}\iso(\mc{E}\otimes\mc{F})[n]$ and $\mc{E}\otimes\mc{F}[n]\iso(\mc{E}\otimes\mc{F})[n]$ in $H^0\mc{A}^\pt$. Now, we show that for objects $\mc{E},\mc{F},\mc{G}$ and a morphism $\phi\in Z^0Hom_{\mc{A}^\pt}(\mc{F},\mc{G})$, if the claim is true for $\mc{E}\otimes\mc{F}$ and $\mc{E}\otimes\mc{G}$, then it is true for 
$\mc{E}\otimes Cone(\phi)$. That is, we have
\begin{eqnarray*}
(\mc{E}\otimes\mc{F})_l &\iso & \bigoplus_{i+j=l}E_i\otimes F_j\\
(\mc{E}\otimes\mc{G})_k &\iso & \bigoplus_{i+p=k}E_i\otimes G_k
\end{eqnarray*}
We have from \cite[Proposition~6.1]{tenssmmot} that in $H^0\mc{A}^\pt$, $\mc{E}\otimes Cone(\phi)\iso Cone(\mc{E}\otimes\mc{F}\to\mc{E}\otimes\mc{G})$. Thus,
\begin{eqnarray*}
(\mc{E}\otimes Cone(\phi))_l &=& (\mc{E}\otimes\mc{G})_l\oplus (\mc{E}\otimes\mc{F})_{l+1}\\
&=& (\oplus_{i+j=l}E_i\otimes F_j)\oplus(\oplus_{i+p=l+1}E_i\otimes G_p)\\
&=& \oplus_{i+k=l}E_i\otimes Cone(\phi)_{l-i}
\end{eqnarray*}
Also, $(\mc{E}\otimes Cone(\phi))_{kl}=cone(\mc{E}\otimes\mc{F}\to\mc{E}\otimes\mc{G})_{kl}=$
\[ \left(\begin{array}{lr}
\oplus_i\big((-1)^{(l-k+1)(l-i)}e_{(i+k-l)\,i}\otimes\id_{G_{l-i}}+(-1)^i\id_{E_i}\otimes g_{(k-i)\,(l-i)}\big) & \hspace*{-7.5cm} \oplus_i(-1)^i\phi^*_{(k-i)\,(l-i+1)}\\
0 & \hspace*{-7.5cm}\oplus_i\big((-1)^{(l-k+1)(l-i+1)}e_{(i+k-l)\,i}\otimes\id_{F_{l-i+1}}+(-1)^i\id_{E_i}\otimes f_{(k-i+1)\,(l-i+1)}\big)
\end{array}\right) \]
which is equal to 
\[ \bigoplus_i\Big[(-1)^{(l-k+1)(l-i)}e_{(i+k-l)\,i}\otimes\id_{Cone(\phi)_{l-i}}
+(-1)^i\id_{E_i}\otimes \left(\begin{array}{cc}g_{(k-i)\,(l-i)}& \phi_{(k-i)\,(l-i+1)}\\0& f_{(k-i+1)\,(l-i+1)}\end{array}\right)\Big] \]
This proves the claim.
\end{proof}
Thus, we have in $H^0\mc{B}^\pt$, \begin{eqnarray*}
X^\pt(\mc{E}\otimes\mc{F})&=&\{(\oplus_{i+j=l}X(E_i\otimes F_j))_l;\\
&&{\ds\oplus_i\left((-1)^{(l-k+1)(l-i)}X(e_{(i+k-l),i}\otimes\id_{F_{l-i}})\right.}\\
&& \left. + (-1)^iX(\id_{E_i}\otimes f_{(k-i),(l-i)})\right)\}\\
\text{and\hspace*{3cm}}&& \\
X^\pt(\mc{E})\otimes X^\pt(\mc{F})&=&\{(\oplus_{i+j=l}X(E_i)\otimes X(F_j))_l;\\
&&{\ds\oplus_i\left((-1)^{(l-k+1)(l-i)}X(e_{(i+k-l),i})\otimes\id_{X(F_{l-i})}\right.}\\
&& \left. + (-1)^i\id_{X(E_i)}\otimes X(f_{(k-i),(l-i)})\right)\}
\end{eqnarray*}
Note that, for objects $\mc{G}:=\{{G_i},g_{ij}:G_j\to G_i\},\,\mc{H}:=\{H_n,h_{mn}:h_n\to h_m\}$ in $\mc{B}^\pt$, if we have maps $\phi_{ii}:G_i\to H_i$ in $H^0\mc{B}$ such that $\phi_{ii}g_{ij}=h_{ij}\phi_{jj}$, then we get
a morphism $\phi:=\{\phi_{ij}\}:\mc{G}\to\mc{H}$ in $H^0\mc{B}^\pt$ defined as $\phi_{ij}=\phi_{ii}$ if $i=j$ and $0$ otherwise. 
Since $H^0X$ is a lax tensor functor, it follows from above that there is a map $\theta:X^\pt(\mc{E})\otimes X^\pt(\mc{F})\to X^\pt(\mc{E}\otimes\mc{F})$ in $H^0\mc{B}^\pt$, that is $H^0X^\pt$ is also a lax tensor functor.

{\it iii.} 
If $X$ is a homotopy tensor functor, by the argument above,  \[X^\pt(\mc{E})\otimes X^\pt(\mc{F})\iso X^\pt(\mc{E}\otimes\mc{F})\] in $H^0\mc{B}^\pt$. Indeed, if the $\phi_{ii}:E_i\to F_i$ are isomorphisms in $H^0\mc{B}$, then so is $\phi:\mc{E}\to\mc{F}$ in $H^0\mc{B}^\pt$.
\end{proof}

\begin{rem}\label{rem:lax}
Note that in the proof of Proposition~\ref{prop:homtens}{\it ii.}, $X$ need not be a pseudo-tensor functor, we only need that $X$ is a DG functor and $H^0X$ is a lax tensor functor.
\end{rem}

\section{The  functor $\mc{K}^b(dgPrCor_S)\to DM^{\eff}(S)$ is lax tensor}
In \cite[\S~6.3]{smmot}, Levine constructs an exact functor 
\[ \rho^\mathrm{eff}_S: SmMot^\eff_{gm}(S)\to DM^\eff(S). \]
The aim in this section is to show that when $S$ is regular semi-local and essentially smooth over
a field of characteristic zero,
 $\rho^\eff_S$ is a lax tensor functor.
 
 Recall from \cite[Lemma~1.9]{smmot} that we can associate a DG category $dg\mc{C}$ to a tensor category $(\mc{C},\otimes)$ with a co-cubical object $\square^*$ with co-multiplication $\delta^*$.
 $dg\mc{C}$ has the same objects as $\mc{C}$, and for objects $X,Y$ in $dg\mc{C}$, $Hom_{dg\mc{C}}(X,Y)^*$
is the non-degenerate complex associated with the cubical abelian group
\[ \ub{n}\mapsto Hom_\mathcal{C}(X\otimes\square^{n},Y). \]
i.e.\  \[ Hom_{dg\mc{C}}(X,Y)^n:= Hom_\mathcal{C}(X\otimes\square^{-n},Y)/\mathrm{degn}. \] (Also see \cite[\S~1.2.5]{tenssmmot}.) 

We also have the extended DG category $\dg\mathcal{C}$, with the same objects as $\mc{C}$. The Hom complexes are defined as follows: For each $m$, we have the non-degenerate complex $Hom_{dg\mathcal{C}}(X,Y\otimes\square^m)^*$; let 
$Hom_{\mc{C}}(X,Y\otimes\square^m)^*_0$ be the subcomplex consisting of $f$   such that
\[
p_{m,i}\circ f=0\in Hom_{\mc{C}}(X,Y\otimes\square^{m-1})^*;\quad i=1,\ldots, m.
\]
Then,
\[
Hom_{\dg\mathcal{C}}(X,Y)^p:=\prod_{m-n=p}Hom_{dg\mc{C}}(X,Y\otimes\square^m)^{-n}_0.
\]
See \cite[\S~1.3.3]{tenssmmot} for details. It is shown in \cite[Theorem~2.4.2]{tenssmmot} that there is a homotopy tensor structure on $\dg\mc{C}$, when $\square^*$ admits bi-multiplication and is homotopy invariant. 
 
Let us assume that $S$ is semi-local and essentially smooth over
a field of characteristic zero. From \cite[Remark~1.4.4]{tenssmmot}, we have that $SmMot^\eff_{gm}(S)$ is equivalent to $\mc{K}^b(dgPrCor_S)^\natural$ where $^\natural$ denotes the idempotent completion. We describe below the construction of Levine in \cite[\hbox{\it loc. cit}]{smmot} to give a functor
\[ \rho^\mathrm{eff}_S: SmMot^\eff_{gm}(S)\to  DM^\eff(S). \]

We also recall from \cite[\S~5.1]{smmot} that for $X, Y$ in $\sm/S$, the cubical Suslin complex $C^S(Y,r)^*(X)$ is the complex associated to the cubical object \[ n\mapsto z^S_{equi}(Y,r)(X\times\square^n_S), \]
where $\square^n_S\simeq\mathbb{A}^n_S$ and $z^S_{equi}(Y,r)(X)$ is the free abelian group on the integral subschemes $W\subset X\times_S Y$
such that the projection $W\to X$ dominates an irreducible component $X'$ of $X$, and such that, for each $x\in X$, the fiber $W_x$ over $x$
has pure dimension $r$ over $k(x)$, or is empty. If $X$ is in $\mathbf{Proj}/S$, then $C^S(X,0)$ is the presheaf $Hom_{dgCor_S}(\blank,X)^*$ on
$\sm/S$.
Sending $U$ in $\sm/S$ to $C^S(X,0)(U)$ gives the object $C^S(X)$ in $C^-(Sh^{tr}_{\mathrm{Nis}}(S))$.

For $X, Y$ in $\mathbf{Proj}/S$, the composition law in $dgCor_S$ defines a natural map
\[ \tilde{\rho}_{X,Y}:Hom_{dgPrCor_S}(X,Y)^*\to Hom_{C(Sh^{tr}_{\mathrm{Nis}}(S))}(C^S(X),C^S(Y)). \]
This defines a DG functor
\[ \tilde{\rho}:dgPrCor_S\to C(Sh^{tr}_{\mathrm{Nis}}(S)).\]
Applying the functor $\mc{K}^b$ and composing with the total complex functor
\[ \Tot:\mc{K}^b(C(Sh^{tr}_{\mathrm{Nis}}(S)))\to K(Sh^{tr}_{\mathrm{Nis}}(S))\]
and the quotient functor
\[ K(Sh^{tr}_{\mathrm{Nis}}(S))\to DM^\eff(S) \]
gives us the exact functor
\[ \mc{K}^b(dgPrCor_S)\to DM^\eff(S). \]
Extending canonically to the idempotent completion, we define the exact functor
\[ \rho^\mathrm{eff}_S: SmMot^\eff_{gm}(S)\to  DM^\eff(S). \]

We also have a map (see \cite[Proposition~1.3.5]{tenssmmot})
\[ \pi_{X,Y}: Hom_{\dg PrCor_S}(X,Y)^* \to Hom_{dgPrCor_S}(X,Y)^* \]
giving us a functor of DG categories 
\[ \hat{\rho}:=\tilde{\rho}\circ\pi: \dg PrCor_S \to C(Sh^{tr}_{\mathrm{Nis}}(S)). \]

We will first show that 
\begin{prop}\label{prop:pstensf}
$\hat{\rho}$ is a pseudo-tensor functor with respect to the pseudo-tensor structure on $\dg PrCor_S$ defined in \cite{tenssmmot} and one induced from the tensor structure $\otimes^{tr}_S$  on $C(Sh^{tr}_{\mathrm{Nis}}(S))$. \end{prop}
The tensor structure $\otimes^{tr}_S$ on $Sh^{tr}_{\mathrm{Nis}}(S)$ is defined as follows: we take the canonical left resolution $\mc{L}(\mc{F})$ of a sheaf
$\mc{F}$ by taking the canonical surjection
\[ \mc{L}_0(\mc{F}):=\bigoplus_{\begin{array}{c}\scriptstyle X\in\sm/S \\ \scriptstyle s\in\mc{F}(X)\end{array}}\Z^{tr}_S(X)\longby{\phi_0}\mc{F} \]
where $\Z^{tr}_S(X):=Cor_S(\blank,X)$ is the representable presheaf for $X\in\sm/S$. Setting $\mc{F}_1:=\ker \phi_0$ and iterating, we get the complex 
$\mc{L}(\mc{F})$. Then, we define
\[ \mc{F}\otimes^{tr}_S\mc{G}:=H_0(\mc{L}(\mc{F})\otimes^{tr}_S\mc{L}(\mc{G})) \]
using the fact that
\[ \Z^{tr}_S(X)\otimes^{tr}_S\Z^{tr}_S(X'):=\Z^{tr}_S(X\times_S X'). \]

\begin{lem}
For $X,X'\in \mathbf{Proj}_S$, there exists a map of presheaves
\[ C^S(X)^m\otimes^{tr}_S C^S(X')^n \to C^S(X\times_S X')^{m+n}. \]
\end{lem}

\begin{proof}
Note that 
\[ \mc{L}_0(C^S(X)^m)=\bigoplus_{\begin{array}{c}\scriptstyle Y\in\sm/S \\ \scriptstyle s\in C^S(X)^m(Y)\end{array}}\Z^{tr}_S(Y)
\longby{\;\phi_0^{X,m}}C^S(X)^m \]
where $Cor_S(Y\times_S\square^*_S,X)_0$ is the non-degenerate complex associated to the cubical abelian group $\ub{n}\mapsto Cor_S(Y\times_S\square^n_S,X)$.

Firstly, we want to produce a map of presheaves
\begin{equation}\label{eqn:Lmap}
\mc{L}_0(C^S(X)^m)\otimes^{tr}_S\mc{L}_0(C^S(X')^n) \to \mc{L}_0(C^S(X\times_S X')^{m+n}).
\end{equation}
We have $\mc{L}_0(C^S(X)^m)\otimes^{tr}_S\mc{L}_0(C^S(X')^n) =$ \\
\begin{align*}
 = & \Big( \bigoplus_{\begin{array}{c}\scriptstyle Y\in\sm/S \\ \scriptstyle s\in C^S(X)^m(Y)\end{array}}\Z^{tr}_S(Y)\Big) \otimes^{tr}_S \Big( \bigoplus_{\begin{array}{c}\scriptstyle Y'\in\sm/S \\ \scriptstyle s'\in C^S(X')^n(Y')\end{array}}\Z^{tr}_S(Y')\Big) \\
=& \bigoplus_{\begin{array}{c}\scriptstyle Y,Y'\in\sm/S \\ \scriptstyle s\in C^S(X)^m(Y)\\ \scriptstyle s'\in C^S(X')^n(Y')\end{array}}\Z^{tr}_S(Y\times_SY')
\end{align*}

Using the map \eqref{eqn:cstimes}, we thus get a map
\[ \xymatrix{ \mc{L}_0(C^S(X)^m)\otimes^{tr}_S\mc{L}_0(C^S(X')^n)\ar[r]&  \ds\bigoplus_{\begin{array}{c}\scriptstyle Y\times_SY'\in\sm/S \\ 
\scriptstyle s\boxtimes s'\in C^S(X\times_S X')^{m+n}(Y\times_SY')\end{array}}\Z^{tr}_S(Y\times_SY') \ar@{^{(}->}[d] \\
\mc{L}_0(C^S(X\times_S X')^{m+n}) \ar@{=}[r] & \ds\bigoplus_{\begin{array}{c}\scriptstyle Y''\in\sm/S \\ 
\scriptstyle t\in C^S(X\times_S X')^{m+n}(Y'')\end{array}}\Z^{tr}_S(Y'') } \]

The map \begin{equation}\label{eqn:cstimes} C^S(X)^m(Y)\otimes C^S(X')^n(Y') \by{\boxtimes} C^S(X\times_S X')^{m+n}(Y\times_SY') \end{equation}
is given as follows:
\[ \xymatrix{ Hom_{dgCor_S}(Y,X)^m\otimes Hom_{dgCor_S}(Y',X')^n \ar[d]^{F_{Y,X}\otimes F_{Y',X'}} \\
Hom_{\dg Cor_S}(Y,X)^m\otimes Hom_{\dg Cor_S}(Y',X')^n \ar@{=}[d] \\
P_1^{\dg Cor_S}(\{Y\},X)^m\otimes P_1^{\dg Cor_S}(\{Y'\},X')^n \ar[d]^-{[\text{By representability of the pseudo-tensor structure }P^{\dg Cor_S}]} \\
P_2^{\dg Cor_S}(\{Y,Y'\},X\times_S X')^{m+n} \ar[d]^{\sum_{a=0}^{m+n}\delta^*_{a,m+n-a}} \\
Hom_{\dg Cor_S}(Y\times_S Y',X\times_S X')^{m+n} \ar[d]^{\pi_{Y\times_S Y',X\times_S X'}} \\
Hom_{dgCor_S}(Y\times_S Y',X\times_S X')^{m+n} } \]
Here, $\delta^*_{a,b}$ is induced by the map $\delta_{a,b}$ defined as the composition
\[ \square^{a+b}\longby{\delta^{a+b}}\square^{a+b}\otimes\square^{a+b}\xrightarrow{p^1_{a,b}\otimes p^2_{a,b}}\square^{a}\otimes\square^{b}. \]

Composing the map \eqref{eqn:Lmap} with $\phi_0^{X\times_S X',m+n}$ gives us a map 
\[ \mc{L}_0(C^S(X)^m)\otimes^{tr}_S\mc{L}_0(C^S(X')^n) \to C^S(X\times_S X')^{m+n}. \]
The fact that this induces a map on $H_0(\mc{L}(C^S(X)^m)\otimes^{tr}_S\mc{L}(C^S(X')^n))$ follows from definition, noting that
\begin{multline*}
(\mc{L}(C^S(X)^m)\otimes^{tr}_S\mc{L}(C^S(X')^n))_1=\\
=\Big(\bigoplus_{\begin{array}{c}\scriptstyle Z,Z'\in\sm/S \\ \scriptstyle t\in \ker(\phi_0^{X,m})(Z)\\ \scriptstyle t'\in C^S(X')^n(Z')\end{array}}\Z^{tr}_S(Z\times_SZ')\Big)\oplus\Big(\bigoplus_{\begin{array}{c}\scriptstyle Z,Z'\in\sm/S \\ \scriptstyle t\in C^S(X)^m(Z)\\ \scriptstyle t'\in \ker(\phi_0^{X',n})(Z')\end{array}}\Z^{tr}_S(Z\times_SZ')\Big) \\
\overset{(\phi_1^{X,m}\otimes^{tr}_S\id) + (\id\otimes^{tr}_S\phi_1^{X',n})}{\xrightarrow{\hspace*{3.4cm}}} \bigoplus_{\begin{array}{c}\scriptstyle Y,Y'\in\sm/S \\ \scriptstyle s\in C^S(X)^m(Y)\\ \scriptstyle s'\in C^S(X')^n(Y')\end{array}}\Z^{tr}_S(Y\times_SY')\\
=(\mc{L}(C^S(X)^m)\otimes^{tr}_S\mc{L}(C^S(X')^n))_0
\end{multline*}
and that $\phi_0^{X\times_S X',m+n}\circ\big((\phi_1^{X,m}\otimes^{tr}_S\id) + (\id\otimes^{tr}_S\phi_1^{X',n})\big)=0$. This completes the proof of the lemma.
\end{proof}

\begin{proof}[Proof of Proposition~\ref{prop:pstensf}]
It directly follows from the lemma that we have a map in $C(Sh^{tr}_{\mathrm{Nis}}(S))$, 
\[ C^S(X)\otimes^{tr}_S C^S(X') \longby{\gamma_{X,X'}} C^S(X\times_S X'). \]
Note that $\gamma_{X,X'}$ is natural in $X$ and $X'$, since so is $\boxtimes$ (in \eqref{eqn:cstimes}).
Thus, we have maps, for $X,Y,Z$ in $\mathbf{Proj}_S$,
\[ \xymatrix{ P_2^{\dg PrCor_S}(\{X,Y\},Z) \ar[r] & Hom_{\dg PrCor_S}(X\times_S Y,Z) \ar[d]^{\pi_{X\times_S Y,Z}} \\
& Hom_{dgPrCor_S}(X\times_S Y,Z) \ar[dl]_{\tilde{\rho}_{X\times_S Y,Z}} \\  Hom_{C(Sh^{tr}_{\mathrm{Nis}}(S))}(C^S(X\times_S Y),C^S(Z)) \ar[r]_-{\gamma^*_{X,X'}} & Hom_{C(Sh^{tr}_{\mathrm{Nis}}(S))}(C^S(X)\otimes^{tr}_S C^S(Y),C^S(Z)).} \]

That this map is compatible with composition follows from the fact that $\gamma_{X,X'}$ is natural in $X$ and $X'$. This shows that $\hat{\rho}$ is a pseudo-tensor functor.
\end{proof}

\begin{cor}
If $S$ is regular semi-local and essentially smooth over
a field of characteristic zero, $\rho^\eff_S$ is a lax tensor functor.
\end{cor}
\begin{proof}
Combining Proposition~\ref{prop:pstensf} with Proposition~\ref{prop:homtens}.{\it ii.} and noting that $\mc{K}^b(\mc{A})$ is equivalent to $H^0\mc{A}^\pt$, we get that $\mc{K}^b(\hat{\rho})$ is a lax tensor functor. Recall from \cite[\S~1.3.3]{tenssmmot}, that there is a quasi-equivalence $dgPrCor_S\by{F}\dg PrCor_S$ such that $\pi\circ F=\id$. Thus, $\tilde{\rho}=\hat{\rho}\circ F$. 
Since the tensor structure on the category $\mc{K}^b(dgPrCor_S)$ is induced from $\mc{K}^b(\dg PrCor_S)$ via the equivalence of categories $\mc{K}^b(F)$, we have that $\mc{K}^b(\tilde{\rho})$ is also a lax tensor functor.

Since $\Tot$ is a tensor functor and the tensor structure $\otimes_S$ on $DM^\eff(S)$ is induced from the tensor product $\otimes^{tr}_S$ via the localization map, we get that 
\[ \mc{K}^b(dgPrCor_S)\to DM^\eff(S) \]
is a lax tensor functor. 
Since extending to idempotent completion preserves tensor structure, it follows from definition that $\rho^\eff_S$ is a lax tensor functor.
\end{proof}

\section{$\rho^\eff_S$ is a tensor functor}

Let $\tilde{ChMot}^\eff_S$ be the category with the same objects as $\mathbf{Proj}_S$ and with morphisms (for $X\to S$ of pure dimension $d_X$
over $S$)
\[ Hom_{\tilde{ChMot}^\eff_S}(X,Y):=CH_{d_X}(X\times_S Y), \]
and let ${ChMot}^\eff(S)$ be its idempotent completion. We have the functor 
\[ c_S: \mathbf{Proj}_S \to ChMot^\eff(S) \]
sending a morphism $f:X\to Y$ to the graph of $f$. In \cite{smmot}, Levine showed that
\begin{prop}[\cite{smmot}, Proposition~5.19]\label{prop:chmot} {\it i}. There is a fully faithful embedding $\psi^\eff_S:ChMot^\eff(S)\to SmMot^\eff_{gm}(S)$ such that
there is a commutative diagram \[
\xymatrix{ \mathbf{Proj}_S \ar[rd] \ar[r]^-{c_S} &  ChMot^\eff(S) \ar[d]^{\psi^\eff_S} \\ & SmMot^\eff_{gm}(S) } \]

\noindent {\it ii}. $\psi^\eff$ extends to a fully faithful embedding $\psi:ChMot(S)\to SmMot_{gm}(S)$. \\

\noindent {\it iii}. The maps $\psi^\eff$ and $\psi$ are compatible with the action of $Cor_S$.
\end{prop}

Note that we have the following commutative diagram of functors

\[ \xymatrix{ \mathbf{Proj}_S \ar[r]_-{M_S} \ar@/^1.5pc/[rr]^{m_S} \ar[rd]_{c_S} & SmMot_{gm}^\eff(S) \ar[r]_-{\rho^\eff_S} & DM^\eff(S) \\
& ChMot^\eff(S) \ar[u]^{\psi^\eff_S} & } \]

where $m_S(X)$ is the image of $\Z^{tr}_S(X)$ in $DM^\eff(S)$. That $m_S=\rho^\eff_S\circ M_S$ follows from \cite[Lemma~6.8]{smmot}. We first show that
\begin{lem} For $X,Y$ in $\mathbf{Proj}_S$, we have
\[ \rho^\eff_S(M_S(X)\otimes M_S(Y))\iso \rho^\eff_S(M_S(X))\otimes_S\rho^\eff_S(M_S(Y)). \]
\end{lem}
\begin{proof}
It follows from Proposition~\ref{prop:chmot}.{\it iii}. that, 
\[ \psi^\eff_S(c_S(X)\otimes c_S(Y))\iso \psi^\eff_S(c_S(X))\otimes\psi^\eff_S(c_S(Y)). \]
Also, since $c_S(X)\otimes c_S(Y) \iso c_S(X\times_S Y)$, we have
\[ M_S(X\times_S Y) \iso M_S(X)\otimes M_S(Y). \]
Composing both sides with $\rho^\eff_S$, we get
\[ m_S(X\times_S Y) \iso \rho^\eff_S(M_S(X)\otimes M_S(Y)). \]
But, $m_S(X\times_S Y)=m_S(X)\otimes_Sm_S(Y)$ by definition. The lemma follows by noting that $m_S=\rho^\eff_S\circ M_S$.
\end{proof}

\begin{thm}
The functors $\rho^\eff_S$ and $\rho_S$ are tensor functors.
\end{thm}
\begin{proof}
In the previous section, we have already shown that $\rho^\eff_S$ is a lax tensor functor. That is, for objects $M, N$ in $SmMot^\eff_{gm}(S)$, there is a morphism in $DM^\eff(S)$
\[ \rho^\eff_S(M)\otimes_S\rho^\eff_S(N) \to \rho^\eff_S(M\otimes N). \]
which is natural in $M$ and $N$. But, by the above lemma, this is an isomorphism when $M=M_S(X)$ and $N=M_S(Y)$ for $X,Y$ in $\mathbf{Proj}_S$. Since $SmMot_{gm}^\eff(S)$ is generated by objects in $\mathbf{Proj}_S$ as an idempotently complete triangulated category, this shows that $\rho_S^\eff$ is a tensor functor. 
Since the composition
\[ SmMot_{gm}^\eff(S) \xrightarrow{\rho_S^\eff} DM^\eff(S)\to DM(S) \]
is a tensor functor that sends $\blank\otimes\mathbb{L}$ to an invertible endomorphism, it factors through a canonical extension
\[ \rho_S: SmMot_{gm}(S) \to DM(S) \]
which is also a tensor functor.
\end{proof}

\begin{cor}
For $X\in \mathbf{Proj}_S$ of dimension $d$, $X\mapsto X^D:=X\otimes \mathbb{L}^{-d}$ gives an exact duality
\[ D: SmMot_{gm}(S)^{op}\ra SmMot_{gm}(S).\]
\end{cor}
\begin{proof}
Note that for $X\in \mathbf{Proj}_S$ of dimension $d$ over $S$, the dual of $m_S(X)$ is $m_S(X)(-d)[-2d]$ in $DM(S)$. Since $\rho_S$
is fully faithful (see \cite[Corollary~6.14]{smmot}) and $\rho_S(X\otimes\mathbb{L}^{-d})=m_S(X)(-d)[-2d]$, we immediately get the duality in $SmMot_{gm}(S)$ as required. Since the dual object in a tensor category is unique upto unique isomorphism, this defines an exact duality involution $D$.
\end{proof}

\bibliography{biblo}{}
\bibliographystyle{amsalpha}
\end{document}